\documentclass[11pt,a4paper]{article}
\usepackage{mathrsfs, amsmath, amsthm, amssymb, bm, mathtools, thm-restate}
\usepackage{graphicx, xcolor, tikz}
\usetikzlibrary{calc, shapes}
\usepackage{caption, subcaption}
\usepackage{array}
\usepackage{cases}
\usepackage{tikz}
\usepackage{float}
\usetikzlibrary{calc}
\usetikzlibrary{arrows}
\usetikzlibrary{patterns}
\usetikzlibrary{shapes,snakes}
\makeatletter
\makeatletter
\def\th@plain{%
	\upshape 
}
\makeatother

\makeatletter
\renewenvironment{proof}[1][\proofname]{\par
	\pushQED{\qed}%
	\normalfont \topsep6\p@\@plus6\p@\relax
	\trivlist
	\item[\hskip\labelsep
	\bfseries
	#1\@addpunct{.}]\ignorespaces
}{%
\popQED\endtrivlist\@endpefalse
}
\makeatother

\newtheorem{theorem}{Theorem}
\newtheorem{lem}{Lemma}
\newtheorem{cor}{Corollary}
\theoremstyle{definition}
\newtheorem{definition}{Definition}

\newtheorem{case}{Case}
\newtheorem{subcase}{Subcase}

\usepackage{chngcntr}
\counterwithin{subcase}{case}

\usepackage[left=25mm, top=25mm, bottom=25mm, right=25mm]{geometry}
\usepackage[T1]{fontenc}
\usepackage[inline]{enumitem}
\usepackage[square, numbers, sort&compress]{natbib}
\usepackage[pdftex,%
bookmarks=true,%
bookmarksnumbered=true, 
bookmarksopen=true, 
plainpages=false,%
pdfpagelabels,%
colorlinks=true, 
linkcolor=blue, 
citecolor=blue,%
anchorcolor=green,
urlcolor=black,
breaklinks=true,
hyperindex=true]{hyperref}

\usepackage[normalem]{ulem}
\begin{document}
	\title{Truncated degree AT-orientations of outerplanar graphs }
	\author{Chenglong Deng and Xuding Zhu\thanks{This research is supported by Grant: NSFC 12371359, U20A2068. } \\
    School of Mathematical Science, Zhejiang Normal University, China}

	\maketitle
	
	\begin{abstract}
 An AT-orientation of a graph $G$ is an orientation $D$ of $G$ such that the number of even Eulerian sub-digraphs and the number of odd Eulerian sub-digraphs of $D$ are distinct. Given a mapping $f: V(G) \to \mathbb{N}$, we say $G$ is $f$-AT if $G$ has an AT-orientation $D$ with $d_D^+(v) < f(v)$ for each vertex $v$. For a positive integer $k$, we say $G$ is $k$-truncated degree-AT if $G$ is $f$-AT for the mapping $f$ defined as $f(v) = \min \{k, d_G(v)\}$. This paper proves that 2-connected  outerplanar graphs other than odd cycles are $5$-truncated degree-AT, and 2-connected bipartite outerplanar graphs are $4$-truncated degree-AT. As a consequence, 2-connected   outerplanar graphs other than odd cycles are $5$-truncated degree paintable, and 2-connected bipartite outerplanar graphs are $4$-truncated degree paintable. 
 This improves the result of Hutchinson in [On list-coloring outerplanar graphs. J. Graph Theory, 59(1):59–74, 2008], where it was proved that maximal 2-connected outerplanar graphs other than $K_3$ are 5-truncated degree-choosable, and 2-connected bipartite outerplanar graphs are 4-truncated degree-choosable. 
    \end{abstract}
\section{Introduction}
  
  Assume $G$ is a graph. A {\em list assignment} 
 of $G$ is a mapping  $L$  that  assigns to each vertex $v$   a set $L(v)$ of permissible colours. Given a list assignment $L$ of $G$, an $L$-colouring of $G$ is a mapping $\phi:V(G)\to\cup_{v\in V(G)} L(v)$ such that for each vertex $v,\phi(v)\in L(v)$ and for any edge $e=uv$ of $G$, $\phi(u)\neq\phi(v)$. We denote by $\mathbb{N}^G$  the set of mappings $f:V(G)\to\mathbb{N}=\{0,1,\ldots,\}$. For $f\in\mathbb{N}^G$, an {\em $f$-list assignment} of $G$ is a list assignment $L$ of $G$ with $|L(v)|\geq f(v)$ for each vertex $v$.  We say $G$ is $f$-choosable if $G$ is $L$-colourable for any $f$-list assignment $L$. The {\em choice  number} $ch( G)$ of $G$ is the minimum $k$ for which $G$ is $k$-choosable. We say $G$ is {\em degree-choosable} if $G$ is $f$-choosable for the function $f$ defined as $f(v)=d_G(v)$ for each vertex $v$.

  For $k \ge 3$, it is NP-hard to determine if a given graph $G$ is $k$-choosable. On the other hand, it is easy to decide if a graph $G$ is degree-choosable. The following result was proved in \cite{MR593902} and \cite{MR498216}.

  \begin{theorem}
      \label{thm-degreechoosable}
      A connected graph $G$ is not degree-choosable if and only if $G$ is a Gallai-tree, i.e., each block of $G$ is a clique or an odd cycle.
  \end{theorem}

The concept of $k$-truncated degree-choosability, which  combines $k$-choosability and degree-choosable, is defined as follows:

 \begin{definition}
	\label{def-tdat}
	Assume $G$ is graph and $k$ is a positive integer. We say $G$ is {\em $k$-truncated degree-choosable}  if $G$ is $f$-choosable, where $f$ is defined as $f(v)=\min\{k,d_G(v)\}$ for $v\in V(G)$.  
\end{definition}

The term ``$k$-truncated degree-choosable" was first used in \cite{ZZZ}, however, the $k$-truncated degree choosability of outerplanar graphs was first studied by Hutchinson \cite{MR2432887} in 2008. Richter asked whether every 3-connected non-complete planar graph is 6-truncated-degree-choosable (see \cite{MR2432887}). 
Note that if $G$ is a Gallai-tree, then $G$ is not degree-choosable, and hence not $k$-truncated degree-choosable for any $k$. If $G=K_{2, k^2}$, then it is also easy to see that $G$ is not $k$-truncated degree-choosable. Some other non-$k$-truncated degree-choosable planar graphs can be obtained from $K_{2, k^2}$ by some simple graph operation, such as subdiving edges, etc. 
Thus to study truncated degree-choosability of planar graphs,   it is natural to restrict to $3$-connected non-complete planar graphs.

  Hutchinson \cite{MR2432887} first studied Richter's problem. Instead of planar graphs, she considered truncated degree choosability of outerplanar graphs. She proved that 2-connected maximal outerplanar graphs  (i.e., each inner face is a triangle) other than $K_3$ are $5$-truncated degree-choosable, and 2-connected bipartite outerplanar graphs are $4$-truncated degree-choosable. 
  
  Note that a graph  $G$ is $k$-truncated degree-choosable does not imply that  a subgraph of $G$ is  $k$-truncated degree-choosable.
  Hutchinson asked whether her result can be extended to all   2-connected  outerplanar graphs, i.e., every 2-connected outerplanar graph other than odd cycles are $5$-truncated degree-choosable.
  This question was answered in \cite{LWZZ},  where truncated degree DP-colourability of $K_{2,4}$-minor free graphs was considered.   The concept of DP-colouring of a graph is a variation of list colouring. For $f \in \mathbb{N}^G$, an {\em $f$-cover } of $G$ is a pair $(L,M)$, where $L=\{L(v): v \in V(G)\}$ is a family of disjoint sets with $|L(v)| \ge f(v)$ for each vertex $v$, and $M=\{M_e:  e\in E(G)\}$ is a family of matchings, where for each edge $e=uv$, $M_e$ is a matching between $L(u)$ and $L(v)$. 
  Given an $f$-cover $(L,M)$ of $G$, an $(L,M)$-colouring of $G$ is a mapping $f: V(G) \to \bigcup_{v \in V(G)}L(v)$ such that $f(v) \in L(v)$ for each vertex $v$, and $f(u)f(v) \notin M_{uv}$ for each edge $uv \in E(G)$. We say $G$ is DP-$f$-colourable if for every $f$-cover $(L,M)$ of $G$, $G$ has a $(L,M)$-colouring. 
  The {\em DP-chromatic number} $\chi_P(G)$ of $G$ is the minimum $k$ such that $G$ is DP-$k$-colourable.
  
  It is easy to see and well known \cite{MR3758240} that if $G$ is DP-$f$-colourable, then $G$ is $f$-choosable, and hence $ch(G) \le \chi_P(G)$. Moreover, the difference $\chi_P(G)-ch(G)$ can be arbitrarily large. For a positive integer $k$, we say that $G$ is $k$-truncated degree-DP-colourable if $G$ is DP-$f$-colourable for the mapping $f$ defined as $f(v) = \min \{k, d_G(v)\}$.  
  
  In \cite{LWZZ} it was shown that 2-connected $K_{2,4}$-minor free graphs other than cycles and complete graphs are 5-truncated degree-DP-colourable. This implies that 2-connected $K_{2,4}$-minor free graphs other than cycles and complete graphs are 5-truncated degree-choosable. In particular, all 2-connected (not necessarily maximal) outerplanar graphs other than odd cycles are 5-truncated degree-choosable.

  The $k$-truncated degree-choosability of $K_5$-minor free graphs was studied in \cite{MR2959285}. Denote by $S_k$ the set of vertices of  degree less than $k$ and denote by $d(S_k)$ the smallest distance between connected components of $G[S_k]$. It was shown in \cite{MR2959285} that   for any $k \ge 3$, there are 3-connected non-complete 
 $K_5$-minor free graphs $G$ with $d(S_k) =2$  that are not $k$-truncated degree-choosable. On the other hand,  for $k \ge 8$, every connected $K_5$-minor free graph $G$ with $d(S_k)=3$ is $k$-truncated degree-choosable, provided that $G$ is not a Gallai-tree; and for $k \ge 7$, every 3-connected non-complete $K_5$-minor free graph $G$ with $d(S_k) \ge 3$ is $k$-truncated degree-choosable.  

  In \cite{ZZZ}, Zhou, Zhu and Zhu studied truncated degree choosability of planar graphs and proper minor closed families of graphs. They constructed a 3-connected noncomplete planar graph that is not 7-truncated degree-choosable, which answers Richter's question in negative (even if $6$ is replaced by $7$). On the other hand, they proved that every 3-connected noncomplete planar graph is 16-truncated degree-choosable. For a proper minor closed family $\mathcal{G}$ of graphs, let $s$ be the minimum integer such that $K_{s,t} \notin \mathcal{G}$ for some $t$. Then there is a constant $k$ such that every $s$-connected graph $G \in \mathcal{G}$ is $k$-truncated degree choosable, provided that $G$ is not a Gallai tree. 
  

In this paper, we extend Hutchinson's results  to online list colouring and Alon-Tarsi orientation. 

The coloration of an online list of a graph $G$ is defined by means of a two-person game. 
Given a graph $G$ and $f \in \mathbb{N}^G$, the {\em $f$-painting game} on $G$ is a game with two players: Lister and Painter. Initially, each vertex $v$ of $G$ is uncoloured and has $f(v)$ tokens. In each round, Lister chooses a set $U$ of uncoloured vertices and takes away one token from each vertex of $U$, and Painter colours an independent set $I$ contained in $U$.  Lister wins the game if there is an uncoloured vertex with no tokens left, and Painter wins the game if every vertex is coloured. We say $G$ is {\em $f$-paintable} if Painter has a winning strategy in the $f$-painting game on $G$. The {\em painter number} $\chi_P(G)$ of $G$ is the minimum integer $k$ for which $G$ is $k$-paintable. We say $G$ is {\em $k$-truncated degree-paintable} if $G$ is $f$-paintable for the function $f$ defined as $f(v)=\min\{k, d_G(v)\}$ for every vertex $v$.

 If  $L$ is an $f$-list assignment of $G$ with $\bigcup_{v \in V(G)}L(v) =[n] $, and in the $i$th round, Lister chooses $U_i = \{v: i \in L(v)\}$, then Painter's winning strategy constructs an $L$-colouring of $G$. So if $G$ is $f$-paintable, then $G$ is $f$-choosable, and $ch(G) \le \chi_P(G)$ for any $G$. The $f$-painting game is an online version of list colouring of $G$, in which the information of the list assignment is revealed in many rounds, and painter needs to make decision on colouring based on partial information of the list assignment. 
 It is known \cite{MR3558040} that the difference $\chi_P(G)-ch(G)$ can be arbitrarily large. 

The techniques used in \cite{MR2432887} and \cite{LWZZ} do not work for paintability. The tool we use in this paper is Alon-Tarsi orientation. We consider truncated degree AT-orientations of outerplanar graphs.   
  
  Assume $D$ is an orientation of a graph $G$. For a subset $H$ of arcs of $D$, $D[H]$ denotes the subdigraph induced by $H$. We say $D[H]$ is an {\em Eulerian subdigraph} if $d^{+}_{D[H]}(v)=d^{-}_{D[H]}(v)$ for each vertex $v$. Let \begin{eqnarray*}
  \mathcal{E}(D)&=&\{H: D[H] \text{  is an Eulerian subdigraph}\}, \\
  \mathcal{E}_e(D)&=&\{H\in \mathcal{E}(D) : |H| \text{ is even}\}, \text{  and }  \mathcal{E}_o(D)=\{H\in \mathcal{E}(D) : |H| \text{  is odd}\}, \\
  {\rm diff}(D) &=&|\mathcal{E}_e(D)|-|\mathcal{E}_o(D)|.
  \end{eqnarray*}
    We say $D$ is AT-orientation if ${\rm diff}(D)\ne 0$.

 For $f\in\mathbb{N}^G$, an {\em $f$-Alon-Tarsi orientation} ( an $f$-AT orientation, for short) of $G$ is an Alon-Tarsi orientation $D$ of $G$ with $d_D^+(v)\leq f(v)-1$ for each vertex $v$.  We say $G$ is $f$-AT if $G$ has an $f$-AT orientation. The {\em Alon-Tarsi number} $AT(G)$ of $G$ is the minimum $k$ such that $G$ is $k$-AT (i.e., $f$-AT for the constant function $f(v)=k$ for all $v$).

 We say $G$ is $k$-truncated degree-AT if $G$ is $f$-AT, where $f(v)=\min\{k, d_G(v)\}$ for every vertex $v$.
 We prove the following results.

\begin{theorem}
	\label{thm-main}
	If $G$ is a  2-connected  outerplanar graph and is not an odd cycle, then $G$ is 5-truncated degree-AT.
\end{theorem}

\begin{theorem}
    \label{thm-bi}
    If $G$ is a  2-connected bipartite outerplanar graph, then $G$ is 4-truncated degree-AT.
\end{theorem}

Alon and Tarsi \cite{MR1179249} proved that if $G$ is $f$-AT, then $G$ is $f$-choosable. 
  Schauz \cite{MR2515754} proved that if $G$ is $f$-AT, then $G$ is $f$-paintable. Therefore,
Theorems \ref{thm-main} and \ref{thm-bi} imply that 2-connected outerplanar graphs other than odd cycles are 5-truncated degree-choosable as well as 5-truncated degree-paintable,
 and 2-connected bipartite outerplanar graphs are 4-truncated degree-choosable as well as 4-truncated degree-paintable. 

  \begin{cor}
      If $G$ is a  2-connected  outerplanar graph and is not an odd cycle, then $G$ is 5-truncated degree-paintable. If $G$ is a  2-connected bipartite outerplanar graph, then $G$ is 4-truncated degree-paintable.
  \end{cor}
	 
Hutchinson's results are tight: There are maximal 2-connected outerplanar graphs other than triangles that are not 4-truncated degree-choosable, and there are bipartite outerplanar graphs that are not 3-truncated degree-choosable. Hence Theorem \ref{thm-main} and Theorem \ref{thm-bi} are also tight.

\section{{Proof of Theorem \ref{thm-main}}}

In our proofs, we need to consider arc-weighted AT-orientations of graphs.
  An arc-weighted orientation of a graph $G$ is a pair $(D,w)$, where $D$ is an orientation of $G$ and $w: E(D) \to \{1,2,\ldots \}$ is a weight function that assigns to each arc $e$ a positive integer $w(e)$ as its weight. The out-degree and in-degree of a vertex $v$ in $(D,w)$ is defined as $d_{(D,w)}^+(v) = \sum_{ e\in E^+(v)}w(e)$ and $d_{(D,w)}^-(v) = \sum_{ e\in E^-(v)}w(e)$. For a subset $H$ of arcs of $D$, $(D[H],w)$ is the arc-weighted sub-digraph induced by arcs in $H$. We say $(D[H],w)$ is an {\em Eulerian subdigraph  } if $d^{+}_{(D[H],w)}(v)=d^{-}_{(D[H],w)}(v)$ for each vertex $v$. 
  Similarly, let
  
  \begin{eqnarray*}
  \mathcal{E}(D,w)&=&\{H: (D[H],w) \text{  is an Eulerian subdigraph}\}, \\
  \mathcal{E}_e(D,w)&=&\{H\in \mathcal{E}(D,w) : |H| \text{ is even}\}, \text{  and }  \mathcal{E}_o(D,w)=\{H\in \mathcal{E}(D,w) : |H| \text{  is odd}\}, \\
  {\rm diff}(D,w) &=&|\mathcal{E}_e(D,w)|-|\mathcal{E}_o(D,w)|.
  \end{eqnarray*}
    We say $(D,w)$ is an arc-weighted AT-orientation if ${\rm diff}(D,w) \ne 0$. In \cite{MR4718223} it was shown that a graph $G$ is $f$-AT if and only if $G$ has an arc-weighted orientation $(D,w)$ such that $d_{(D,w)}^+(v) < f(v)$.

 Let $G$ be an outerplane graph, which is an outerplanar graph with fixed embedding. Our goal is to construct an arc-weighted $5$-truncated degree  AT-orientation $(D,w)$ of $G$. Given an arc-weighted orientation $(D,w)$, it is usually difficult to check whether $(D,w)$ is an AT-orientation. 
 
Assume $G$ is a 2-connected outplanar graph other than an odd cycle. We shall prove that $G$ has an arc-weighted AT-orientation $(D,w)$ with $d_{(D,w)}^+(v) \le \min \{4, d_G(v)-1\}$ for every vertex $v$.
We shall prove this result by using induction. As in many inductive proofs, we shall prove a stronger and more technical result. Indeed, the main difficulty in proving Theorem \ref{thm-main} is to formulate the ``right and more technical" statement to be proved by induction. Once the right statement is formulated, the proof is more or less routing checking.

We first define two terms needed for stating the more technical result. 

\begin{definition}
    \label{def-pairs}
    We denote by $\mathcal{Q}$ the set of pairs $(G,S)$ such that one of the following holds:
    \begin{enumerate}
        \item $G$ is   a 2-connected outerplanar graph and $S$ is a set of boundary edge of $G$. Moreover, if  $G$ is cycle, then $|S|\geq 1$.
        \item $G=K_2$,    $S$ consists of one or two parallel edges connecting the two vertices of $G$.
    \end{enumerate}
\end{definition}

 \begin{definition}
     \label{def-orientation}
     Assume $(G,S) \in \mathcal{Q}$. A {\em valid orientation} of $(G,S)$ is a pair $((D,w), \vec{S})$, where  $(D,w)$ is an arc weighted orientation of $G$, and   $\vec{S}$ is an orientation of $S$,  for which the following hold:
     \begin{enumerate}
         \item $|\mathcal{E}(D,w)|$ is odd.
         \item For each vertex $v$,
         $$d^{+}_{(D,w)}(v)\le \min\{4 - 2d_{\vec{S}}^+(v),d_{G}(v)-1 +d_{\vec{S}}^-(v) \}.$$
      \end{enumerate}
 \end{definition}

We shall prove that every $(G,S) \in \mathcal{Q}$ has a valid orientation $((D,w),\vec{S})$. 
Observe that Theorem \ref{thm-main} follows this result: If $G$ is an even cycle, then let $D$ be the orientation of $G$ that is a directed cycle. Otherwise, $G$ is not a cycle, and  a valid orientation $((D,w), \vec{S})$  with $S = \emptyset$ gives the required arc-weighted orientation of $G$. Note that if   $|\mathcal{E}(D,w)|$ is odd, then certainly $|\mathcal{E}_e(D,w)| \ne  |\mathcal{E}_o(D,w)|$, and hence $(D,w)$ is an AT-orientation. 

The set $\vec{S}$ of oriented boundary arcs of $G$ acts as a ``bank system".  Our ultimate goal is to find an arc-weighted AT-orientation $(D,w)$ of $G$ so that for each vertex $v$,
$$d^{+}_{(D,w)}(v)\le \min\{4,d_{G}(v)-1 \}.$$
If $uv \in S$, then it means that   $G$ is obtained from the original outerplanar graph by removing an ear-chain (to be defined later) with base edge $uv$.
If $(u,v) \in \vec{S}$, then 
the out-degree requirement for $u$ is more restrictive, and the out-degree requirement for $v$ is less restrictive. We can view it as $v$ ``borrowed" some ``out-degree restriction quota" from $u$. Eventually, the removed ear-chain with base edge $uv$ will be put back, and by orienting the edges in the ear-chain, $v$ will pay back the debts, and every vertex satisfies the requirement that  
$d^{+}_{(D,w)}(v)\le \min\{4,d_{G}(v)-1 \}.$

The conclusion that $|\mathcal{E}(D,w)|$ be odd is stronger than $(D,w)$ is an AT-orientation. We need this stronger property of the orientation to carry out the induction proof. 

\begin{lem}
    \label{lem-main}
    Every $(G,S) \in \mathcal{Q}$ admits a valid orientation.
\end{lem}
\begin{proof} 
The proof is  by   induction on the number of vertices of $G$. 

Assume first that $G=K_2$ and $V(G)=\{u,v\}$.    Let $D$ be the orientation of $G$ consisting arc $(v,u)$, and let  $w((v,u))=1$. The only Eulerian subdigraph is the empty subdigraph. So $|\mathcal{E}(D,w)|$ is odd.

If $S$ consists of one edge $vu$, then let  $\vec{S}=\{(u,v)\}$.
If $S$ consists of two edges connecting $u$ and $v$, then 
  $\vec{S} $ consists of two copies of the arc $(u,v)$. 
  In any case,   $d^{+}_{\vec{S}}(u) \le  2$, $d^-_{\vec{S}}(u) = 0$ and $d^-_{\vec{S}}(v) \ge 1$. 
  
Hence 
 
    $$d^{+}_{(D,w)}(u)=0 \le 4-2 d^{+}_{\vec{S}}(u),  \ \text{ and } \  d^{+}_{(D,w)}(v) =1 \le d_G(v)-1 +  d^-_{\vec{S}}(v).$$ 

 Assume next that $G$ is a cycle $(v_1,v_2,\cdots,v_n)$. As $S \ne \emptyset$, we may assume $v_1v_n\in S$. Let $(D,w)$ be the arc-weighted orientation of $G$ with arcs  $(v_i,v_{i+1}) $ for $i=1,2,\cdots, n-1$ and $(v_1,v_n)$, and with all arcs having weight 1.

The only Eulerian subdigraph is the empty subdigraph. So $|\mathcal{E}(D,w)|$ is odd. 

The edge $v_1v_n \in S$ is oriented as $(v_n,v_1)$. For $i=1,2,\ldots, n-1$, if 
 $v_iv_{i+1}\in S$, then it is oriented as an arc $(v_{i+1},v_i)$ in $ \vec{S}$. 

Note that  $v_n$ is the only vertex which may have out-degree 2 in $\vec{S}$ and $d_{(D,w)}^+(v_n)=0$.  Every vertex $v$ has $d_{(D,w)}^+(v) \le 1 = d_G(v)-1$, except that $d_{(D,w)}^+(v_1) = d_G(v_1)=2$. 
As $d_{\vec{S}}(v_1) \ge 1$, we conclude that $((D,w),\vec{S})$ is a valid orientation of $(G,S)$.

Assume $(G,S) \in \mathcal{Q}$, $G \ne K_2$ and $G$ is not a cycle.

\begin{definition}
    An {\em ear} $H$ of $G$ is a cycle $(u_1, u_2,\cdots, u_r)$ such that $d_G(u_i)=2$ for $i \in \{2, 3, \cdots, r-1\}$.
    The edge $u_1u_r$ is the root edge of $H$.
    \end{definition}

    \begin{definition}
    An  {\em  ear-chain} $F$ of $G$ consists of an induced cycle $C=(v_1, v_2,\cdots, v_s)$ and a sequence   $H_1, H_2, \cdots , H_t$ of ears ($1\leq t \leq s-1$), such that for $1 \le i \le t$,  the root edge of $H_i$ is an edge $v_{j_i} v_{j_i+1}$ of the cycle $C$, where $1 \le j_1 < j_2 < \ldots < j_t < s$. Moreover, each vertex $v_i \ne v_1, v_s$ are not incident to any other edges of $G$.
    The edge $v_1v_s$ is the base edge of $F$.   
\end{definition}

For an ear-chain $F$ consisting a cycle $C=(v_1, v_2,\cdots, v_s)$ and ears $H_1, H_2, \cdots , H_t$, 
We denote by $$E(F) = E(C) \cup (\bigcup_{i=1}^tE(H_i)) - \{v_1v_s\}.$$

 \begin{figure}[H]
    \centering
    \includegraphics[width=0.7\textwidth]{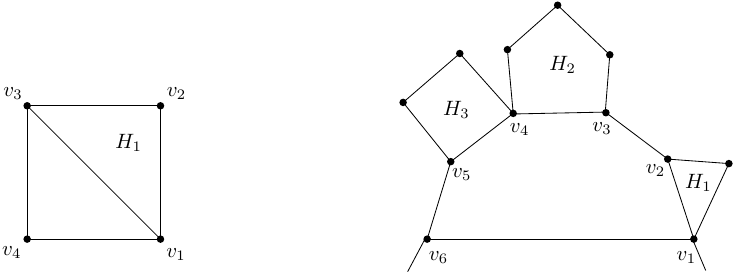}\label{fig:ex}
    \caption{Two example  ear-chains. }
\end{figure}

As $G$ is a 2-connected outerplanar graph and $G$ is not a cycle, it is easy to verify that $G$ contains an ear chain $F$. 

Assume $F$ is an ear chain in $G$.
Assume  $F$ consists an induced cycle $C=(v_1,v_2,\ldots, v_s)$ with base edge $v_1v_s$, and ears $H_1,H_2,\ldots, H_t$. For $j=1,2,\ldots, t$, $H_j$ is an induced cycle $ (u_{j,1},u_{j,2}, \ldots, u_{j, m_j})$ with root edge $u_{j,1}u_{j,m_j} = v_{i_j}v_{i_j+1}$, where $1 \le i_1 < i_2 < \ldots < i_t \le s-1$.
(Note that the two end vertices of the root edge of $H_j$ have two labels: $u_{j,1} = v_{i_j}$ and $u_{j, m_j} = v_{i_j+1}$. They are both vertices of the cycle $C$ and the ear $H_j$.) 
Thus
$$E(F) = \{v_iv_{i+1}: i =1,2,\ldots, s-1\} \cup \bigcup_{j=1}^t \{u_{j,l}u_{j,l+1}: l=1,2\ldots, m_j-1\}.$$

  Let $G' = G - (V(F) - \{v_1,v_s\})$, and  $S' = (S \cap E(G')) \cup   \{v_1v_s\}$.

By induction hypothesis, there is a valid  orientation $((D',w'), \vec{S}')$ of $(G', S')$. Without loss of generality,  we may assume $(v_1,v_s)\in \vec{S'}$.

It is possible that $G'=K_2$. In this case, $v_1v_s$ is a boundary edge of $G$ and possibly an edge in $S$. In this case, $S'$ consists of a pair of parallel edges connecting $v_1$ and $v_s$.

Let $S_F = S \cap E(F)$. We shall construct an arc-weighted orientation $(D_F, w_F)$ of $E(F)$ and an orientation $\vec{S}_F$ of $S_F$.
Let $(D,w) = (D',w') \cup (D_F, w_F)$ and $\vec{S} = \vec{S}' \cup \vec{S}_F - \{(v_1,v_s)\}$. Then we prove that $((D,w), \vec{S})$ is a valid orientation of $(G,S)$. 

First we define a {\em canonical orientation } $((D_F,w_F),\vec{S}_F)$ of $E(F)$ and $S_F$ as follows: 

 The arc set of $D_F$ is
   $$\{(v_i,v_{i+1}): i=1,2,\cdots, s-1\} \cup (\bigcup_{j=1}^t \{(u_{j,l},u_{j,l+1}): l=1,2,\ldots, m_j-1\})$$
   and each arc has weight 1. 

   The canonical orientation is illustrated in Figure \ref{figcase12}(a).

   In other words, all the paths (in $E(F)$) connecting $v_1$ and $v_s$  are oriented as   directed paths from $v_1$ to  $v_s$. 

In $\vec{S}_F$, all the edges   of $S_F$ are oriented in the same direction as in the canonical orientation in $D_F$. 
 
In all the cases below, $\vec{S}_F$ remains the same, and the orientations $ (D_F,w_F) $ will be slight modifications of the canonical orientation. 
To define $(D_F, w_F)$, we shall only specify which edges are oriented differently from the canonical orientation. 

To prove that $((D,w),\vec{S}) $ is a valid orientation of $(G,S)$, we need to show that $|\mathcal{E}(D,w)|$ is odd and all the vertices satisfy the out-degree constraint.

Note that $$\mathcal{E}((D',w') \subseteq \mathcal{E}(D,w).$$
Let
$$\mathcal{E}^{new}(D,w) = \mathcal{E}(D,w) - \mathcal{E}((D',w').$$
As $|\mathcal{E}((D',w')|$ is odd, it remains to prove that $|\mathcal{E}^{new}(D,w)| $ is even. 
 
For each vertex  $v \in V(G')$,
by induction hypothesis, 
\begin{equation}
\label{eq1}
    d^{+}_{(D',w')}(v)\le \min\{4 - 2d_{\vec{S}'}^+(v),d_{G'}(v)-1 +d_{\vec{S}'}^-(v) \}. 
\end{equation} 
We need to show that 
\begin{equation}
\label{eq2}
    d^{+}_{(D,w)}(v)\le \min\{4 - 2d_{\vec{S}}^+(v),d_{G}(v)-1 +d_{\vec{S}}^-(v) \}. 
\end{equation} 
For this purpose, we shall compare the differences on both sides of (\ref{eq1}) and (\ref{eq2}). Let  
\begin{eqnarray*}
    l(v) &=& d^{+}_{(D,w)}(v) - d^{+}_{(D',w')}(v),\\
    r(v) &=& \min\{4 - 2d_{\vec{S}}^+(v),d_{G}(v)-1 +d_{\vec{S}}^-(v) \} - \min\{4 - 2d_{\vec{S}'}^+(v),d_{G'}(v)-1 +d_{\vec{S}'}^-(v) \}.
\end{eqnarray*}
  
To prove that (\ref{eq2}) holds, 
it suffices to show that $l(v) \le r(v)$. For $v \in V(G') - V(F)$, we have $l(v)=r(v)=0$. Note that 
$V(G') \cap V(F)=\{v_1, v_s\}$. In all the cases below, it follows directly from the definition that $l(v_s)=0$. As $d_{\vec{S}}^+(v_s) = d_{\vec{S'}}^+(v_s)$, 
$d_{\vec{S}}^-(v_s) = d_{\vec{S'}}^-(v_s) -1 $ and $d_G(v_s) \ge d_{G'}(v_s)+1$, we conclude that $r(v_s) \ge 0$. Therefore to prove that $((D,w),\vec{S})$ is a valid orientation of $(G,S)$, it suffices to show the following:
\begin{enumerate}
    \item[(1)] $|\mathcal{E}^{new}(D,w)|$ is even.
    \item[(2)] $l(v_1) \le r(v_1)$.
    \item[(3)] For $v \in V(F) - \{v_1, v_s\}$, $$d^{+}_{(D,w)}(v)\le \min\{4 - 2d_{\vec{S}}^+(v),d_{G}(v)-1 +d_{\vec{S}}^-(v) \}. $$
\end{enumerate}

\begin{case}  
     $v_1v_2 $ is not the root edge of $H_1$ and  $v_1v_2 \notin S$. 

   In this case, $(D_F,w_F), \vec{S})$ is exactly the canonical orientation of $E(F)$ and $S_F$.

(1) By definition, $F$ contains at least one ear $H_1$. The edges of $H_1$ are oriented as two directed paths, say $P_1 = (u_{1,1}, u_{1, m_1})$ and $P_2 = (u_{1,1},u_{1,2}, \ldots, u_{1,m_1})$, from $u_{1,1}$ to $u_{1, m_1}$. 
Note that any Eulerian sub-digraph $H \in \mathcal{E}^{new}(D,w)$ contain the arc $(v_1,v_2)$.
Hence $H$ contains exactly one of the two directed paths. If $H \in \mathcal{E}^{new}(D,w)$ contains $P_1$, then replace $P_1$ by $P_2$, we obtain another Eulerian sub-digraph $H' \in \mathcal{E}^{new}(D,w)$, and vice versa. Therefore elements of  $\mathcal{E}^{new}(D,w)$ come in pairs, and $|\mathcal{E}^{new}(D,w)|$ is even.

(2) It follows from definition that $l(v_1)=1$. 
   As $d_G(v_1)=d_{G'}(v_1)+1$, 
$d_{\vec{S}}^+(v_1) = d_{\vec{S}'}^+(v_1)-1$, and 
$d_{\vec{S}}^-(v_1) = d_{\vec{S}'}^-(v_1)$, we conclude that $r(v_1) \ge 1$. So $l(v_1) \le r(v_1)$.

(3) For each vertex   $v\in  V(F)-\{v_1,v_s\}$, 
$d_{\vec{S}}^+(v) \le 1$ and 
$d_{(D,w)}^+(v) \le 2$. 
Hence $d_{(D,w)}^+(v) \le 4 - 2 d_{\vec{S}}^+(v)$.
On the other hand, $d_{(D,w)}^-(v) \ge 1$, and hence 
 $d_{(D,w)}^+(v) \le d_{G}(v)-1$.
Thus $d^{+}_{(D,w)}(v)\le \min\{4 - 2d_{\vec{S}}^+(v),d_{G}(v)-1 +d_{\vec{S}}^-(v) \}.$

This completes the proof of Case 1. 
\end{case}

\begin{case}
     $v_1v_2 $ is not the root edge of $H_1$ and  $v_1v_2 \in S$. 
\begin{subcase}
    $v_2v_3$ is a root edge of $H_1$ and $u_{1,1}u_{1,2} \in S$

    In this case,  $(D_F,w_F)$  is obtained from the canonical orientation by reversing the directions of $v_1v_2$ and $u_{1,1}u_{1,2}$, and these two arcs still have weight 1. This orientation is illustrated in   Figure \ref{figcase12} (b).  

   (1)  Observe that any arc in an Eulerian subdigraph is contained in a directed cycle. As no arc in $E(F)$ is contained in a  directed cycle, $\mathcal{E}^{new}(D,w) = \emptyset$ (and hence $|\mathcal{E}^{new}(D,w)| $ 
    is even).

    (2)  It follows from definition that $l(v_1)=0$. 
      As $d_G(v_1)=d_{G'}(v_1)+1$, $d_{\vec{S}}^+(v_1) = d_{\vec{S}'}^+(v_1)$, and $d_{\vec{S}}^-(v_1) = d_{\vec{S}'}^-(v_1)$, we conclude that $r(v_1) \ge 0$. So $l(v_1) \le r(v_1)$.

    (3)    For every vertex $v\in (V(F)-\{v_1,v_s\})$,   
       $d_{\vec{S}}^+(v) \le 1$ and $d_{(D,w)}^+(v) \le 2$. Hence $d_{(D,w)}^+(v) \le 4 - 2 d_{\vec{S}}^+(v)$. 

       For every vertex $v\in (V(F)-\{v_1,v_s, u_{1,2}\})$,    $d^{+}_{(D,w)}(v)\le d_{G}(v)-1$. As  
       $d_{\vec{S}}^-(u_{1,2}) \ge 1$ and $d_{(D,w)}^+(u_{1,2}) =d_G(u_{1,2})$, we conclude that for every vertex $v\in (V(F)-{v_1,v_s})$, 
       $$d^{+}_{(D,w)}(v)\le \min\{4 - 2d_{\vec{S}}^+(v),d_{G}(v)-1 +d_{\vec{S}}^-(v) \}.$$
   
    This completes the proof of Subcase 2.1.
            
   \end{subcase}
  \begin{subcase}
    $v_2v_3$ is not the root edge of $H_1$ or $v_2v_3$ is the root edge of $H_1$ and $u_{1,1}u_{1,2} \notin S$.

    In this case, $(D_F,w_F)$ is obtained from the canonical orientation by reversing the direction of edge $v_1v_2$, and this arc has  weight 1.

(1) Again, all the arcs in $E(F)$ are not contained in any directed cycles, and hence $\mathcal{E}^{new}(D,w) = \emptyset$.

 (2)  It follows from definition that $l(v_1)=0$. 
      As $d_G(v_1)=d_{G'}(v_1)+1$, $d_{\vec{S}}^+(v_1) = d_{\vec{S}'}^+(v_1)$, and $d_{\vec{S}}^-(v_1) = d_{\vec{S}'}^-(v_1)$, we conclude that $r(v_1) \ge 0$. So $l(v_1) \le r(v_1)$.

 (3)    For  $v\in (V(F)-\{v_1,v_2, v_s\})$, $d_{(D,w)}^+(v) \le \min\{2,  d_{G}(v)-1\}$ and $d_{\vec{S}}^+(v) \le 1$. Hence
     $d_{(D,w)}^+(v) \le  \min\{4 - 2d_{\vec{S}}^+(v),d_{G}(v)-1 +d_{\vec{S}}^-(v) \}.$ 

     If $v_2v_3$ is not the root edge of $H_1$, then for the same reason, $d_{(D,w)}^+(v_2) \le  \min\{4 - 2d_{\vec{S}}^+(v_2),d_{G}(v_2)-1 +d_{\vec{S}}^-(v_2) \}.$
     
    If $v_2v_3$ is the root edge of $H_1$, then   $d^{+}_{(D,w)}(v_2)=3$ and $d_{\vec{S}}^+(v_2)=0$. Hence we also have $d_{(D,w)}^+(v_2) \le  \min\{4 - 2d_{\vec{S}}^+(v_2),d_{G}(v_2)-1 +d_{\vec{S}}^-(v_2) \}.$

     This completes the proof of Subcase 2.2.
   \end{subcase}
 \end{case}
 
 \begin{figure}[H]
    \centering
    \includegraphics[width=1\textwidth]{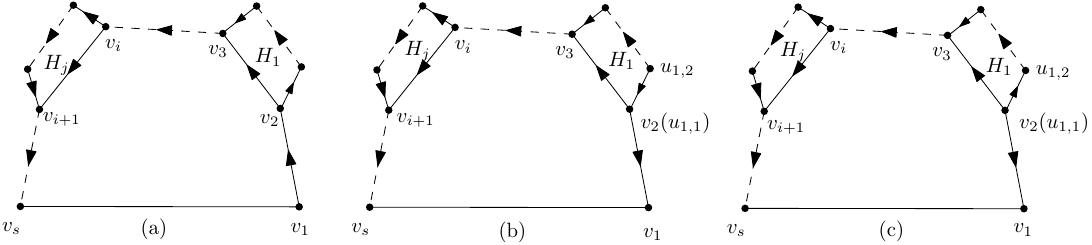}\label{fig:ot1}
    \caption{Orientations $D_F$ for Case 1 and Case 2. }
    \label{figcase12}
\end{figure}

\begin{case} 
     $v_1v_2 $ is the root edge of ear $H_1 $ and $u_{1,1}u_{1,2} \notin S$.  
    \begin{subcase} 
        $v_2v_3$ is the root edge of $H_2$ and $u_{2,1}u_{2,2} \in S$.

       In this case,  $(D_F,w_F)$ is obtained  from the canonical orientation by reversing the direction of edges    $v_1v_2$ and $u_{2,1}u_{2,2}$, and also change the weight of $(u_{1,1}, u_{1,2})$ to $2$. The other arcs have weight 1. See Figure \ref{figcase3} (a).
       
(1)       The arcs  $(u_{1,1}, u_{1,2})$  is not contained in any Eulerian subdigraph, as $w((u_{1,1}, u_{1,2}))=2$ and $d_{(D,w)}^+(u_{1,2})=1$. No arc in $D_F - \{(u_{1,1}, u_{1,2})\}$ is   contained in a  directed cycle  of $(D, w)-\{(u_{1,1},u_{1,2})\}$. Hence $\mathcal{E}^{new}(D,w) = \emptyset$.
       
(2)  It follows from definition that $l(v_1)=2$ . 
        As $d_G(v_1)=d_{G'}(v_1)+2$, $d_{\vec{S}}^+(v_1) = d_{\vec{S}'}^+(v_1)-1$, and $d_{\vec{S}}^-(v_1) = d_{\vec{S}'}^-(v_1)$, we conclude that $r(v_1) \ge 2$.  So $l(v_1) \le r(v_1)$.
   
  (3)     For every vertex $v\in V(F)-\{v_1,v_s\}$,  
       $d_{\vec{S}}^+(v) \le 1$ and $d_{(D,w)}^+(v) \le 2$. Hence $d_{(D,w)}^+(v) \le 4 - 2 d_{\vec{S}}^+(v)$. 
       
       For every vertex $v\in V(F)-\{v_1,v_s, u_{2,2}\}$,    $d^{+}_{(D,w)}(v)\le d_{G}(v)-1$. As  
       $d_{\vec{S}}^-(u_{2,2}) \ge 1$ and $d_{(D,w)}^+(u_{2,2}) =d_G(u_{2,2})$, we conclude that for every vertex $v\in V(F)-\{v_1,v_s\}$, 
       $$d^{+}_{(D,w)}(v)\le \min\{4 - 2d_{\vec{S}}^+(v),d_{G}(v)-1 +d_{\vec{S}}^-(v) \}.$$

       This completes the proof of Subcase 3.1.
         
    \end{subcase} 
    \begin{subcase} 
        $v_2v_3$ is not the root edge of $H_2$, or
        $v_2v_3$ is the root edge of $H_2$ and $u_{2,1}u_{2,2} \notin S$.

         In this case,  $(D_F,w_F)$ is obtained  from the canonical orientation by reversing the direction of edge    $v_1v_2$, and also change the weight of $(u_{1,1}, u_{1,2})$ to $2$. The other arcs have weight 1. See Figure \ref{figcase3} (a).
         
     (1)    Again, the arcs  $(u_{1,1}, u_{1,2})$  is not contained in any Eulerian subdigraph, and 
 no arc in $D_F - \{(u_{1,1}, u_{1,2)}\}$ is   contained in a directed cycle  of $(D, w)-\{(u_{1,1},u_{1,2})\}$. Hence $\mathcal{E}^{new}(D,w) = \emptyset$.

(2)   It follows from the definition that $l(v_1)=2$. 
        As $d_G(v_1)=d_{G'}(v_1)+2$, $d_{\vec{S}}^+(v_1) = d_{\vec{S}'}^+(v_1)-1$, and $d_{\vec{S}}^-(v_1) = d_{\vec{S}'}^-(v_1)$, we conclude that $r(v_1) \ge 2$. So $l(v_1) \le r(v_1)$.

        (3) 
      For  $v\in (V(F)-\{v_1,v_2, v_s\})$, $d_{(D,w)}^+(v) \le \min\{2,  d_{G}(v)-1\}$ and $d_{\vec{S}}^+(v) \le 1$. Hence
      $d_{(D,w)}^+(v) \le  \min\{4 - 2d_{\vec{S}}^+(v),d_{G}(v)-1 +d_{\vec{S}}^-(v) \}.$ 

      If $v_2v_3$ is not the root edge of $H_2$, then for the same reason, $d_{(D,w)}^+(v_2) \le  \min\{4 - 2d_{\vec{S}}^+(v_2),d_{G}(v_2)-1 +d_{\vec{S}}^-(v_2) \}.$
     
      If $v_2v_3$ is the root edge of $H_2$, then   $d^{+}_{(D,w)}(v_2)=3$ and $d_{\vec{S}}^+(v_2)=0$. Hence we also have $d_{(D,w)}^+(v_2) \le  \min\{4 - 2d_{\vec{S}}^+(v_2),d_{G}(v_2)-1 +d_{\vec{S}}^-(v_2) \}.$

        This completes the proof of Subcase 3.2
         
    \end{subcase} 
\end{case} 
\begin{figure}[H]
    \centering
    \includegraphics[width=1\textwidth]{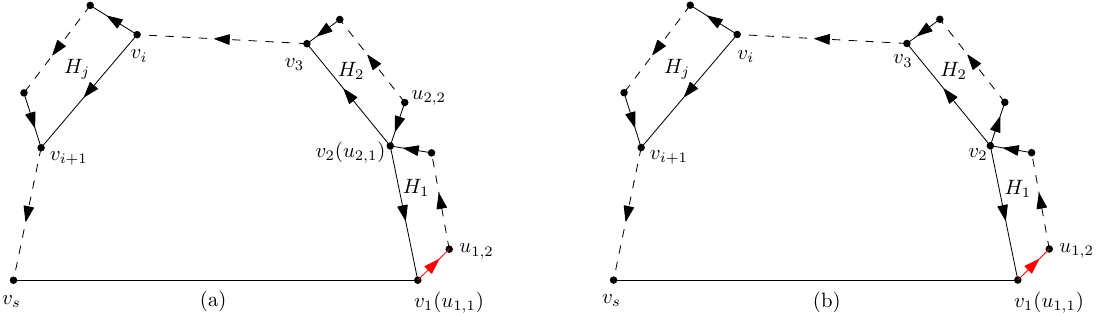}
    \caption{\label{fig:ot2} Orientations $D_F$ for Case 3. A red arc indicates an arc of weight 2.}
    \label{figcase3}
\end{figure}

\begin{case} 
     $v_1v_2 $ is the root edge of ear $H_1 $ and $u_{1,1}u_{1,2} \in S$.
      \begin{subcase} 
        $v_2v_3$ is the root edge of $H_2$ and $u_{2,1}u_{2,2} \in S$.

        In this case,  $(D_F,w_F)$  is obtained from the canonical orientation by reversing the directions of $v_1v_2$, $u_{1,1}u_{1,2}$ and $u_{2,1}u_{2,2}$ and these three arcs still have weight 1. This orientation is illustrated in   Figure \ref{figcase4} (a).  
        
      (1) As no arc in $E(F)$ is contained in a  directed cycle, $\mathcal{E}^{new}(D,w) = \emptyset$.

      (2) It follows from definition that $l(v_1)=0$. 
      As $d_G(v_1)=d_{G'}(v_1)+2$, $d_{\vec{S}}^+(v_1) = d_{\vec{S}'}^+(v_1)$, and $d_{\vec{S}}^-(v_1) = d_{\vec{S}'}^-(v_1)$, we conclude that $r(v_1) \ge 0$. 
   
     (3) For every vertex $v\in (V(F)-\{v_1,v_s, u_{1,2},u_{2,2}\})$, we have 
       $d_{\vec{S}}^+(v) \le 1$ and $d_{(D,w)}^+(v) \le \min\{2,d_{G}(v)-1\}$. Hence   $$d^{+}_{(D,w)}(v)\le \min\{4 - 2d_{\vec{S}}^+(v),d_{G}(v)-1 +d_{\vec{S}}^-(v) \}.$$
      
       For every vertex $v\in\{u_{1,2},u_{2,2}\}$, $d^{+}_{(D,w)}(v)= d_{G}(v)=2$ and $d_{\vec{S}}^-(v) \ge 1$ . Hence   $d^{+}_{(D,w)}(v)\le \min\{4 - 2d_{\vec{S}}^+(v),d_{G}(v)-1 +d_{\vec{S}}^-(v) \}.$

        This completes the proof of Subcase 4.1.

    \end{subcase} 
    \begin{subcase} 
        $v_2v_3$ is not the root edge of $H_2$, or
        $v_2v_3$ is the root edge of $H_2$ and $u_{2,1}u_{2,2} \notin S$. 

       In this case,  $(D_F,w_F)$  is obtained from the canonical orientation by reversing the directions of $v_1v_2$ and $u_{1,1}u_{1,2}$ and these two arcs still have weight 1. This orientation is illustrated in   Figure \ref{figcase4} (b). 

        (1) As no arc in $E(F)$ is contained in a  directed cycle, $\mathcal{E}^{new}(D,w) = \emptyset$.

         (2) It follows from definition that $l(v_1)=0$ . 
      As $d_G(v_1)=d_{G'}(v_1)+2$, $d_{\vec{S}}^+(v_1) = d_{\vec{S}'}^+(v_1)$, and $d_{\vec{S}}^-(v_1) = d_{\vec{S}'}^-(v_1)$, we conclude that $r(v_1) \ge 0$. 

      (3) For  $v\in (V(F)-\{v_1,v_2,u_{1,2}, v_s\})$, $d_{(D,w)}^+(v) \le \min\{2,  d_{G}(v)-1\}$ and $d_{\vec{S}}^+(v) \le 1$.   Hence
      $d_{(D,w)}^+(v) \le  \min\{4 - 2d_{\vec{S}}^+(v),d_{G}(v)-1 +d_{\vec{S}}^-(v) \}.$ 
      
      For $u_{1,2}$, $d_{(D,w)}^+(u_{1,2})=2$, $d_{\vec{S}}^+(u_{1,2}) \le 1$ and $d_{\vec{S}}^-(u_{1,2})\ge 1$. Then   
      $d_{(D,w)}^+(u_{1,2}) \le  \min\{4 - 2d_{\vec{S}}^+(u_{1,2}),d_{G}(u_{1,2})-1 +d_{\vec{S}}^-(u_{1,2}) \}.$

      If $v_2v_3$ is not the root edge of $H_2$, then for the same reason, $d_{(D,w)}^+(v_2) \le  \min\{4 - 2d_{\vec{S}}^+(v_2),d_{G}(v_2)-1 +d_{\vec{S}}^-(v_2) \}.$
     
      If $v_2v_3$ is the root edge of $H_2$, then   $d^{+}_{(D,w)}(v_2)=3$ and $d_{\vec{S}}^+(v_2)=0$. Hence we also have $d_{(D,w)}^+(v_2) \le  \min\{4 - 2d_{\vec{S}}^+(v_2),d_{G}(v_2)-1 +d_{\vec{S}}^-(v_2) \}.$
      
      This completes the proof of Subcase 4.2

    \end{subcase}  
\end{case} 
\end{proof}

\begin{figure}[H]
    \centering
    \includegraphics[width=1\textwidth]{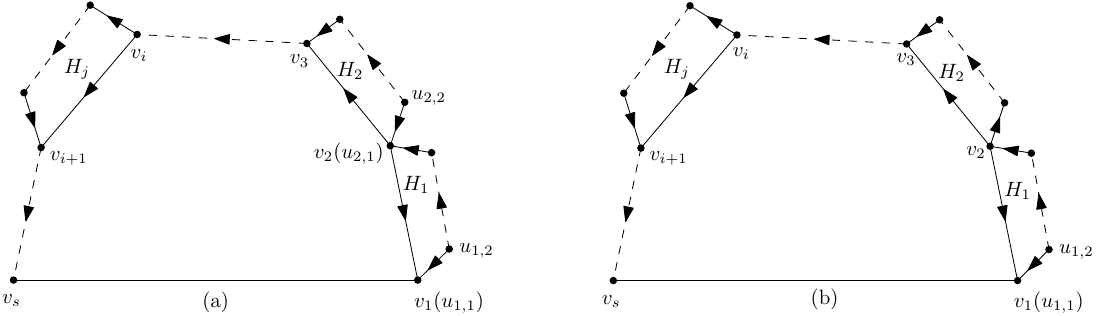}
    \caption{\label{fig:ot2} Orientations $D_F$ for Case 4. }
    \label{figcase4}
\end{figure}

\section{{Proof of Theorem \ref{thm-bi}}}

\begin{definition}
    \label{def-pairs_bi}
    We denote by $\mathcal{B}$ the set of pairs $(G,S)$ such that one of the following holds:
    \begin{enumerate}
        \item $G$ is   a 2-connected bipartite outerplanar graph and $S$ is a set of boundary edge of $G$, where $S$ maybe empty set.
        \item $G=K_2$,    $S$ consists of one or two parallel edges connecting the two vertices of $G$.
    \end{enumerate}
\end{definition}

\begin{definition}
     \label{def-orientation_bi}
     Assume $(G,S) \in \mathcal{B}$. A {\em valid orientation} of $(G,S)$ is a pair $(D, \vec{S})$, where $D$ is an orientation of $G$, and $\vec{S}$ is an orientation of $S$,  for which the following hold:

      For each vertex $v$,
      $$d^{+}_{D}(v)\le \min\{3- d_{\vec{S}}^+(v),d_{G}(v)-1 +d_{\vec{S}}^-(v) \}. $$
     
 \end{definition}

\begin{lem}
    \label{lem-bipartite AT}
    If $G$ is a bipartite graph, then every orientation $D$ of $G$ in an AT-orientation.
\end{lem}

\begin{lem}
    \label{lem-main_bi}
    Any $(G,S) \in \mathcal{B}$ admits a valid orientation.
\end{lem}

Observe that Theorem \ref{thm-bi} follows from Lemma \ref{lem-main_bi} and Lemma \ref{lem-bipartite AT}: Apply Lemma \ref{lem-main_bi} with $S = \emptyset$. 

\begin{proof}

We prove Lemma \ref{lem-main_bi} by   induction on the faces of $G$.

Assume first that $G=K_2$ and $V(G)=\{u,v\}$.    Let $D$ be the orientation of $G$ consisting arc $(v,u)$ The only Eulerian subdigraph is the empty subdigraph. So $D$ is an AT-orientation.

If $S$ consists of one edge $vu$, then let  $\vec{S}=\{(u,v)\}$.
If $S$ consists of two edges connecting $u$ and $v$, then $\vec{S} $ consists of two copies of the arc $(u,v)$.

In any case, $d^{+}_{\vec{S}}(u)\le  2$, $d^-_{\vec{S}}(u) = 0$ and $d^-_{\vec{S}}(v) \ge 1$. As $d^{+}_{D}(u)=0$ and $d^{+}_{D}(v)=d_G(v)=1$, $(D, \vec{S})$ is a valid orientation of $(G,S)$.

We may assume that lemma hold for $|f(G)|<n$. Then we consider $|f(G)|=n$, and  $S$ is a set of boundary edge of $G$. $G$ must contain an ear, which is a $2i$-sided face, $i\geq2$, $(v_1,v_2,\cdots,v_{2i})$, where $d(v_j)=2$ for $j=2,\cdots,2i-1$, and $d(v_1),d(v_{2i})\geq3$. 

Let $S_f=S\cap \{v_1v_2,v_2v_3,\cdots,v_{2i-1}v_{2i}\}$, $G':=G-\{v_2,\cdots,v_{2i-1}\}$, $S':=S-S_f+\{v_1v_{2i}\}$. So $S'$ is a set of boundary edge of $G'$ and $f(G')<n$, by induction hypothesis, there is a valid orientation $(D',\vec{S}')$ of $(G',S')$. Since $v_1v_{2i} \in S'$ , so we may assume $(v_1,v_{2i})\in \vec{S'}$. 

For each vertex  $v \in V(G')$,
by induction hypothesis, 
\begin{equation}
\label{eq3}
    d^{+}_{D'}(v)\le \min\{3 - d_{\vec{S}'}^+(v),d_{G'}(v)-1 +d_{\vec{S}'}^-(v) \}. 
\end{equation} 
We need to show that 
\begin{equation}
\label{eq4}
    d^{+}_{D}(v)\le \min\{3 - d_{\vec{S}}^+(v),d_{G}(v)-1 +d_{\vec{S}}^-(v) \}. 
\end{equation} 
For this purpose, we shall compare the differences on both sides of (\ref{eq3}) and (\ref{eq4}). Let  
\begin{eqnarray*}
    l(v) &=& d^{+}_{D}(v) - d^{+}_{D'}(v),\\
    r(v) &=& \min\{ 3- d_{\vec{S}}^+(v),d_{G}(v)-1 +d_{\vec{S}}^-(v) \} - \min\{3- d_{\vec{S}'}^+(v),d_{G'}(v)-1 +d_{\vec{S}'}^-(v) \}.
\end{eqnarray*}
  
To prove that (\ref{eq4}) holds, 
it suffices to show that $l(v) \le r(v)$.

And again, it suffices to show the following:
\begin{enumerate}
    \item For $v\in\{v_1,v_{2i}\}$, $l(v)\le r(v)$.
    \item For $v\in\{v_2,v_3,\cdots,v_{2i-1}\}$, $$d^{+}_{D}(v)\le \min\{3 - d_{\vec{S}}^+(v),d_{G}(v)-1 +d_{\vec{S}}^-(v) \}.$$  
\end{enumerate}

   If $v_1v_2\notin S$, Let $D:=D'\cup \bigcup_{j=1}^{2i-1}(v_j,v_{j+1})$ .
  For $j=2,\ldots, 2i-1$, if $v_jv_{j+1}\in S_f$, then it is oriented as an arc $(v_{j},v_{j+1})$ in $ \vec{S}_f$. Let $\vec{S}:=\vec{S}'\cup \vec{S}_f \setminus (v_1,v_{2i})$.

  (1) It follows directly from the definition that $l(v_1)=1$ $l(v_{2i})=0$. As $d_{\vec{S}}^+(v_1) = d_{\vec{S}'}^+(v_1)-1$, $d_{\vec{S}}^-(v_{1}) = d_{\vec{S}'}^-(v_{1})$ and $d_G(v_{1})= d_{G'}(v_{1})+1$, we conclude that $r(v_1) \ge 1$.
  And
  $d_{\vec{S}}^+(v_{2i}) = d_{\vec{S}'}^+(v_{2i})$, 
$d_{\vec{S}}^-(v_{2i}) = d_{\vec{S}'}^-(v_{2i}) -1 $ and $d_G(v_{2i}) \ge d_{G'}(v_{2i})+1$, we conclude that $r(v_{2i}) \ge 0$.

 (2) Note that for $j=2,\cdots,2i-1$, $d^{+}_D(v_j)\le d_{G}(v_j)-1$ and $d^-_{\vec{S}}(v_j)\ge 1$ , furthermore $d^+_D(v_j)=1$  and $d_{\vec{S}}^+(v_j) \leq 1$. Hence, $d^{+}_{D}(v_j)\le \min\{3 - d_{\vec{S}}^+(v_j),d_{G}(v_j)-1 +d_{\vec{S}}^-(v_j) \}.$

If $v_1v_2 \in S$, Let $D:=D'\cup (v_2,v_1)\bigcup_{j=1}^{2i-1}(v_j,v_{j+1})$ .
 For $j=1,2,\ldots, 2i-1$, if $v_jv_{j+1}\in S_f$, then it is oriented as an arc $(v_{j},v_{j+1})$ in $ \vec{S}_f$. Let $\vec{S}:=\vec{S}'\cup \vec{S}_f \setminus (v_1,v_{2i})$.

 (1) It follows directly from the definition that $l(v_1)=0$ $l(v_{2i})=0$. As $d_{\vec{S}}^+(v_1) = d_{\vec{S}'}^+(v_1)$, $d_{\vec{S}}^-(v_{1}) = d_{\vec{S}'}^-(v_{1})$ and $d_G(v_{1})= d_{G'}(v_{1})$, we conclude that $r(v_1) \ge 0$.
  And
  $d_{\vec{S}}^+(v_{2i}) = d_{\vec{S}'}^+(v_{2i})$, 
$d_{\vec{S}}^-(v_{2i}) = d_{\vec{S}'}^-(v_{2i}) -1 $ and $d_G(v_{2i}) \ge d_{G'}(v_{2i})+1$, we conclude that $r(v_{2i}) \ge 0$.

 (2) Note that for $j=3,\cdots,2i-1$, $d^{+}_D(v_j)\le d_{G}(v_j)-1$ and $d^-_{\vec{S}}(v_j)\ge 1$. furthermore $d^+_D(v_j)=1$  and $d_{\vec{S}}^+(v_j) \leq 1$. For $v_2$, $d^+_D(v_2)=d_G(v_2)=2$, $d^-_{\vec{S}}(v_2)=1$, $d^+_{\vec{S}}(v_2)\le 1$.
 
 Hence for $j=2,\cdots,2i-1$, $d^{+}_{D}(v_j)\le \min\{3 - d_{\vec{S}}^+(v_j),d_{G}(v_j)-1 +d_{\vec{S}}^-(v_j) \}.$
 
In conclusion, we get a valid orientation $(D,\vec{S})$ of $(G,S)$.
    
\end{proof}

\end{document}